\documentclass{amsart}
\usepackage{graphicx} % Required for inserting images
\usepackage{amssymb,euscript,tikz,units,amsmath}
\usepackage[colorlinks,citecolor=blue,linkcolor=red]{hyperref}
\usepackage{verbatim}
\usepackage[nameinlink]{cleveref}
\usepackage{colonequals}
\usepackage{tikz}
\usepackage[titletoc]{appendix}
\usepackage{yhmath}
\usepackage{enumerate}

\newtheorem{theorem}{Theorem}

\newtheorem{proposition}[theorem]{Proposition}
\newtheorem{lemma}[theorem]{Lemma}

\newtheorem*{que*}{Question}
\newtheorem*{rem*}{Remark}
\usepackage{fancyhdr} % Headers and footers
\usepackage{lastpage}
\begin{document}
\title{Cheeger type inequalities associated with isocapacitary constants on Riemannian manifolds with boundary}

\author{Bobo Hua and Yang Shen  }

\maketitle
{{\bf Abstract:}
In this paper, we study the Steklov eigenvalue of a Riemannian manifold $(M,g)$ with smooth boundary. For compact $M$, we establish a Cheeger-type inequality for the first Steklov eigenvalue by the isocapacitary constant. For non-compact $M$, we estimate the bottom of the spectrum of the Dirichlet-to-Neumann operator by the isocapacitary constant.
}
 \tableofcontents
\section{Introduction}
For an $n$-dimensional compact orientable Riemannian manifold $(M,g)$ with smooth boundary $\partial M$, the Steklov problem on $(M,g)$ is defined as
$$\begin{cases}
    \Delta f(x)=0&\text{ for } x\in M;\\
    \frac{\partial f}{\partial \nu}(x)=\sigma f(x)&\text{ for }x\in\partial M.
\end{cases}$$
where $\Delta$ is the Laplace-Beltrami operator on $(M,g)$ and $\frac{\partial}{\partial \nu}$ is the outward normal derivative along $\partial M$. This problem was firstly introduced in the Euclidean space by Vladimir Steklov in 1902, one may refer to \cite{ste-survey} for more details. The Steklov problem on $(M,g)$ coincides with the spectrum of the DtN operator $\mathcal{D}_M$ which is defined as 
\begin{align*}
\mathcal{D}_M:H^{\frac{1}{2}}(\partial M)\to H^{-\frac{1}{2}}(\partial M)
\end{align*}
$$f\mapsto\mathcal{D}_Mf=\frac{\partial H_f}{\partial \nu},$$
where $H_f$ is the unique harmonic extension of $f$ on $M$. It is well-known that $\mathcal{D}_M$ is a self-adjoint and non-negative operator, its spectrum is discrete, and the eigenvalues could be enumerated as 
$$0=\sigma_0(M)<\sigma_1(M)\leq\sigma_2(M)\leq\cdots.$$
Similar to the case of Laplacian eigenvalues, the Weyl's law also holds for Steklov eigenvalues (see e.g \cite{ste-sur-17}), that is  
$$\#\{j;\sigma_j(M)<\sigma\}=\frac{\textnormal{Vol}(\mathbb{B}^{n-1})\textnormal{Vol}(\partial M)}{(2\pi)^{n-1}}\sigma^{n-1}+O(\sigma^{n-2}),\ \sigma\to \infty,$$
 where $\mathbb{B}^{n-1}$ is the Euclidean ball in $\mathbb{R}^{n-1}$. For the problems of determining the domain which maximizes the first Steklov eigenvalue under some constrains, the readers may refer to \cite{Br01,FS11,GP12,We54}.
For the estimates of the first Steklov eigenvalue by geometry quantities, the readers may refer to \cite{CE11,Es97,Es99,Es00}. Colbois-Girouard-Raveendran \cite{CGR18} studied the Steklov problem on some discretizations of manifolds. 
With the help of random hyperbolic surfaces theory, Han-He-Hong \cite{HHH22} studied the behavior of the first Steklov eigenvalues in moduli space $\mathcal{M}_{g,1}(L)$ and obtained a construction of some special hyperbolic surfaces. For the readers who are interested in random surfaces theory, may refer to \cite{Mir13,WX22,WX-PGT,HSWX23,SW23} for more examples. The Steklov eigenvalues of graphs have also been well studied, see e.g. \cite{HM20,HH22,HHW17,Pe19,Pe21,SY22} for more results.

Similar to the Cheeger constant $h_c$ and famous Cheeger's inequality, Jammes \cite{Ja15} introduced an isoperimetric constant $h_j$ for the Steklov eigenvalue and proved the following Cheeger type inequality,
$$\sigma_1(M)\geq\frac{1}{4}h_c(M)h_j(M).$$
According to the result of Buser in \cite{Bu82}, if $\textnormal{Ric}(M)\geq -(n-1)a^2$, then
\begin{align}\label{l-bu}
\lambda_1(M)\leq 2a(n-1)h_c+10h_c^2.
\end{align}
However, similar upper bound does not hold for Steklov eigenvalue, for which one refers to Page $39$ of \cite{ste-sur} for a counter example. It is natural to ask the following question.
\begin{que*}[Open Question 4.6 in \cite{ste-sur}]
    Can one define a different Cheeger-type isoperimetric constant $h^\prime$ for which $\sigma_1$ satisfies a Buser-type inequality as in \eqref{l-bu}?
\end{que*}
In this article, we answer this question partially by the isocapacitary constant.

Assume $\Omega$ is a subdomain of an $n$-dimensional Riemannian manifold $M$. For any $F\subset\Omega$, the capacity $\textnormal{Cap}(F,\Omega)$ of $F$ relative to $\Omega$ is defined by
$$\textnormal{Cap}(F,\Omega):=\inf\left\{\int_{\Omega}|\nabla u|^2dx:u\in C_0^\infty(\Omega),u\geq 1\text{ on }F\right\}.$$
Maz'ya \cite{ma62, ma64, ma05, ma09} estimated eigenvalues of the Laplacian operator by capacity. He proved that the first Dirichlet eigenvalue $\lambda_1(\Omega)$  satisfies the following inequalities
\begin{align}\label{in-mazya}\frac{1}{4}\Gamma(\Omega)\leq\lambda_1(\Omega)\leq\Gamma(\Omega),\end{align}
where 
$$\Gamma(\Omega):=\inf\limits_{\substack{F\subset\Omega\\ \text{compact} }}\frac{\textnormal{Cap}(F,\Omega)}{|F|}.$$
Joint with M$\ddot{\text{u}}$nch and Wang, the first author \cite{HMW24} obtained similar estimates on graphs and applied them to the discrete Steklov problem. 

Inspired by the results of Maz'ya, we will establish the corresponding inequalities between the first Steklov eigenvalue and the isocapacitary constant. 
Assume $(M,g)$ is a compact $n$-dimensional Riemannian manifold with smooth boundary $\partial M$. Let $V_M$ be the volume form induced by $g$ and $d\sigma$ be the $(n-1)$ dimensional Hausdorff measure on $\partial M$. For any measurable subset $U\subset\partial M$, denote by 
$$m(U)=\int_U 1d\sigma.$$
For any compact subsets $A,B\subset M$, define 
$$\textnormal{Cap}(A,B,M):=\inf\left\{\int_{\Omega}|\nabla u|^2dV_M:\begin{matrix}u\in C^\infty(M),u\geq 1\text{ on }A,\\ u\leq 0\text{ on }B\end{matrix}\right\}.$$
Moreover, define
    $$\Gamma_{\partial}( M)=\min\limits_{A,B\subset\partial M}\frac{\textnormal{Cap}(A,B)}{\min\{m(A),m(B)\}},$$
    where $A,B$ are taken over all compact subsets of $\partial M$. Similar to \eqref{in-mazya}, we prove the main result of this paper.
\begin{theorem}\label{mt-1}
    Assume $(M,g)$ is a compact Riemannian manifold with smooth boundary $\partial M$, then 
    $$\frac{1}{4}\Gamma_{\partial}( M)\leq \sigma_1( M)\leq 2\Gamma_{\partial}( M).$$
\end{theorem}

For the case that $M$ is non-compact,  although the spectrum of $\mathcal{D}_M$ may not be discrete, it still lies on positive real axis and 
\begin{align*}
\inf\textnormal{Spec}(\mathcal{D}_{M})=\inf\limits_{f}\frac{\int_{M}|\nabla H_f|^2dV_{ M}}{\int_{\partial M}f^2d\sigma},
\end{align*}
where $f$ is taken over all functions in $C_c^\infty(\partial M)$ and $H_f$ is a harmonic extension of $f$ on $M$, see Subsection \eqref{s-2.2} for definition. 
For any compact subset $F\subset\partial M$, define 
$$\textnormal{Cap}(F, M)=\inf\limits_{f}\int_{ M}|\nabla f|^2dV_{ M},$$
where $f$ is taken over all smooth functions such that 
$$f\equiv 1\text{ on $F$ and }f\in C_c^\infty(M).$$ Also define
$$\Gamma_\partial( M)=\inf\limits_{\substack{F\subset\partial M\\ \text{ compact}}}\frac{\textnormal{Cap}(F, M)}{m(F)}.$$
Then we have the following result.
\begin{theorem}\label{mt-2}
    Assume $(M,g)$ is a non-compact Riemannian manifold with embedded smooth boundary $\partial M$, then 
    $$\frac{1}{4}\Gamma_{\partial}( M)\leq\inf\textnormal{Spec}(\mathcal{D}_{ M})\leq 2\Gamma_{\partial}( M).$$
\end{theorem}

In the last section, we will consider some special Riemannian manifolds. As a direct application of Theorem \ref{mt-1}, we estimate the first Steklov eigenvalues of compact hyperbolic surfaces with geodesic boundaries. For the case of $(n+1)$-dimensional  upper half unit ball $\mathbb{D}^{n+1}_+$ in the hyperbolic space $\mathbb{D}^{n+1}\ (n\geq 1)$, which are non-compact Riemannian manifolds with smooth embedded boundary, we give an exact expression of corresponding Dirichlet-to-Nuemann operator and estimate the bottom of its spectrum.
\section{Preliminaries}

\subsection{Mixed boundary condition}
Assume $(X,g)$ is a compact Riemannian manifold with smooth boundary $\partial X$ and $Y$ is a compact subset of $\partial X$. Now we consider $H^{\frac{1}{2}}_Y(\partial X)$, which is a subspace of $H^{\frac{1}{2}}(\partial X)$ defined by 
$$H^{\frac{1}{2}}_Y(\partial X)=\left\{f\in H^{\frac{1}{2}}(\partial X);\ \text{$f\equiv 0$ on $Y$}\right\}.$$ Recall that the Dirichlet-to-Neumann operator $\mathcal{D}_{X}$ on $X$ is defined by 
\begin{align*}
    \mathcal{D}_{X}: H^{\frac{1}{2}}(\partial X)\to H^{-\frac{1}{2}}(\partial X),
\end{align*}
$$f \mapsto  \frac{\partial H_f}{\partial \nu},$$
where $H_f$ is the unique harmonic extension of $f$ on $ X$ and $\frac{\partial}{\partial \nu}$ is the outward normal derivative along $\partial X$.
Restrict $\mathcal{D}_X$ on $H^{\frac{1}{2}}_Y(\partial X)\subset H^{\frac{1}{2}}(\partial X)$, the corresponding eigenvalues coincide with the eigenvalues of the following Steklov-Dirichlet problem
\begin{align}\label{s-d-p}
\begin{cases}
    \Delta u=0 &\text{ in }X;\\
    \frac{\partial u}{\partial \nu}=\xi u&\text{ on }\partial X\setminus Y;\\
    u=0&\text{ on }Y.
\end{cases}
\end{align}
Also assume such eigenvalues could be enumerated as
$$0<\xi_1(X, Y)\leq \xi_2(X, Y)\leq\cdots.$$
Their variational characterisation is given by (see e.g. \cite{ste-sur})
$$\xi_k(X, Y)=\min\limits_{E\in\mathcal{E}_0(k)}\max\limits_{0\neq u\in E}\frac{\int_{X}|\nabla u|^2dV_X}{\int_{\partial X}|u|^2d\sigma},$$
where $dV_X$ is the volume form on $X$ induced by $g$ and $d\sigma$ is the relative Hausdorff measure on $\partial X$. Moreover $\mathcal{E}_0(k)$ consists of all $k$-dimensional subspaces of 
$$H^1_Y(X)=\{f\in H^1(X);\ f\equiv 0\text{ on }Y\}.$$
In particular, one may conclude that
\begin{align}\label{e-ray-11}
    \xi_1(X, Y)=\min\limits_{u\in H_Y^1(X)}\frac{\int_{X}|\nabla u|^2dV_X}{\int_{\partial X}|u|^2d\sigma}.
\end{align}
For calculations of mixed Steklov-Dirichlet eigenvalues of some special Riemannian manifolds, one may refer to \cite{ste-sur, CV21, FS19} for examples.

Assume $(M,g)$ is a Riemannian manifold with smooth boundary $\partial M$ and $N$ is a compact submanifold of $M$ with smooth boundary $\partial N$. Define the interior boundary $\partial^I N$ and exterior boundary $\partial^E N$ of $N$ by 
$$\partial^I N=\partial N\cap\textnormal{int}( M)\text{ and }\partial^E N=\partial N\cap\partial M$$
respectively. One may see the illustration in Figure \ref{fig:01}, the blue part is the exterior boundary of $N$ and the red part is the interior boundary of $N$.
 \begin{figure}[ht]
\centering
\includegraphics[width=0.6\textwidth]{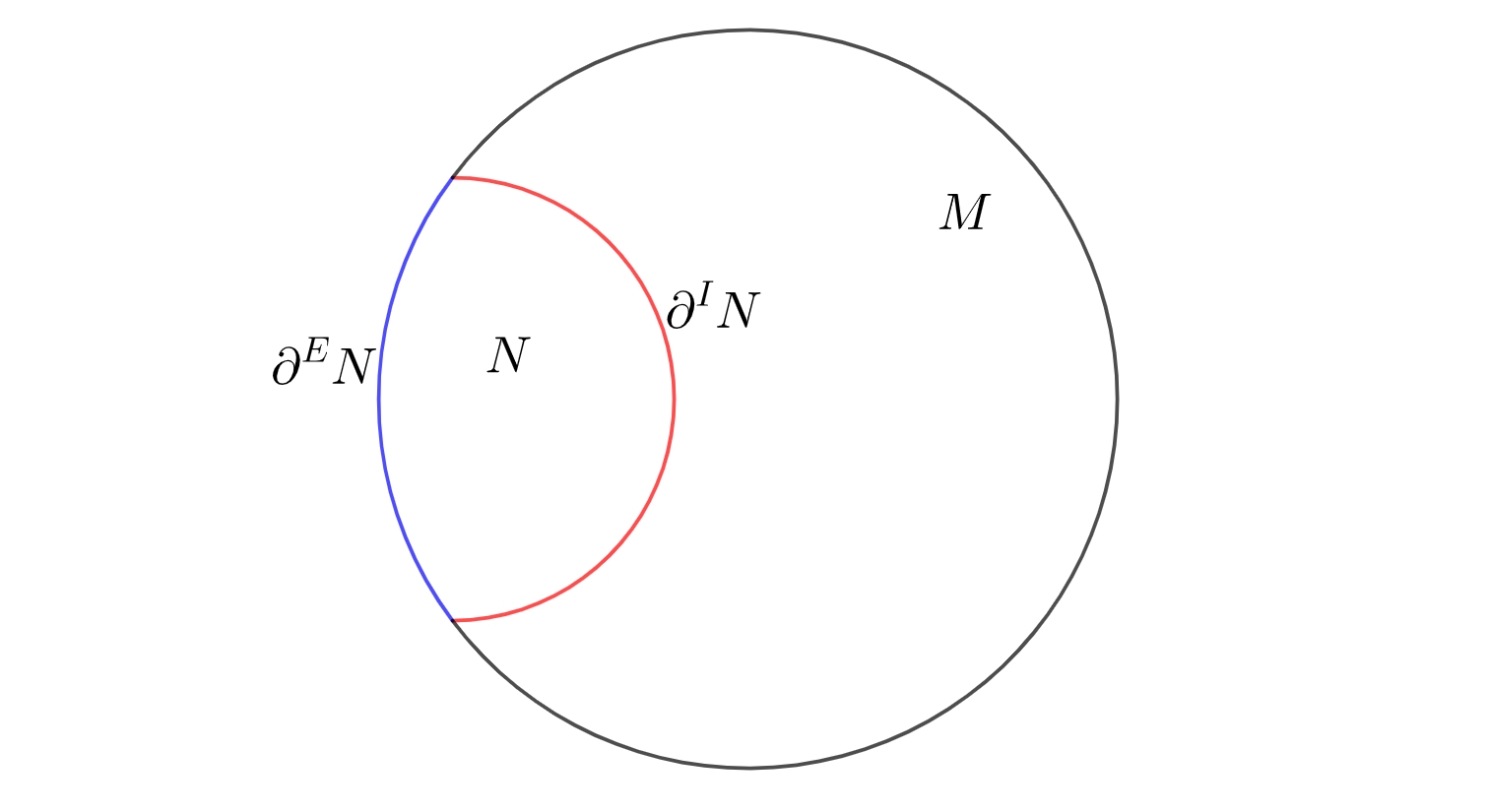}
\caption{}
\label{fig:01}
\end{figure}

In this article, we mainly consider the Steklov-Dirichlet problem \eqref{s-d-p} for the case that $X=N$ and $Y=\partial^I N $. The equality \eqref{e-ray-11} tells that
\begin{align}\label{e-ray}
    \xi_1\left(N,\partial^I N\right)=\min\limits_{u\in H^1_{\partial^I N}(N)}\frac{\int_{N}|\nabla u|^2dV_N}{\int_{\partial N}|u|^2d\sigma}.
\end{align}

\subsection{Non-compact Riemannian manifold}\label{s-2.2}
In this subsection, we consider the case of non-compact Riemannian manifold with embedded  smooth boundary. We firstly recall the definition of DtN operator on such Riemannian manifold.

Assume $(M,g)$ is a non-compact Riemannian manifold with embedded smooth boundary. For any function $f\in C_c^\infty(\partial M)$, there is a unique function $H_f\in H^1(M)$ defined as follows. Set $$f^+=\frac{1}{2}(f+|f|)\text{ and }f^-=\frac{1}{2}=(f-|f|),$$
then we have 
$$f=f^++f^-\text{ and }f^+\geq 0,\ f^-\leq 0\text{ on }\partial M.$$
Take a point $p\in\partial M$. For any $r>0$, denote by $$M_r=M\cap \mathbb{B}(p,r)$$
where $\mathbb{B}(p,r)$ is the geodesic ball with center $p$ and radius $r$. Let $H_{f_r^+}$ be the function on $M_r$ such that 
$$\begin{cases}
    \Delta H_{f_r^+}=0 &\text{in }M_r;\\
    H_{f_r^+}=0&\text{on }\partial^I M_r;\\
    H_{f_r^+}=f^+&\text{on }\partial^E M_r.
\end{cases}$$
Then $H_{f^+_r}\geq 0$ in $M_r$ for any $r>0$. There exists $r_0>0$ such that for any $r>r_0$,
$$\textnormal{Supp}(f)\subset\partial^E M_r.$$
Hence for $0<r_1<r_2$, $H_{f_{r_2}}$ is also a function on $M_{r_1}$ and 
$$\begin{cases}
    \Delta \left(H_{f_{r_2}^+}-H_{f_{r_1}^+}\right)=0 &\text{in }M_{r_1};\\
    H_{f_{r_2}^+}-H_{f_{r_1}^+}\geq 0&\text{on }\partial M_{r_1}.
\end{cases}$$
It follows that 
$$H_{f_{r_2}^+}\geq H_{f_{r_1}^+}\text{ on }M_{r_1}.$$
On the other hand, from maximum principle, we have for any $r>0$, 
\begin{align*}
\max\limits_{p\in M_r}H_{f_r^+}(p)\leq\max\limits_{p\in\partial M_r}f^+(p)\leq \max\limits_{p\in\partial M}f^+(p)<\infty.
\end{align*}
Then one may define a function $H_{f^+}$ on $M$ as 
$$H_{f^+}(q)=\lim\limits_{r\to\infty}H_{f_{r}^+}(q)\text{ for any }q\in M.$$
One easily checks that $H_{f^+}$ is not dependent on the choice of point $p$. For any point $q\in\partial M$, let $U$ be a small neighborhood of $q$. Then for some $r_0$ large enough, $\left\{H_{f_r^+}\right\}_{r\geq r_0}$ is sequence of harmonic functions on $U$, which are uniformly bounded and equicontinuous (the gradient $\left|\nabla H_{f_r^+}\right|$ is uniformly bounded). It follows that $H_f^+$ is Lipschitz continuous on $U$ and harmonic on $U\setminus\partial M$, and hence is smooth up to the boundary by the regularity theory of elliptic equations, this yields that the normal derivative $\frac{\partial H_f}{\partial \nu}(q)$ exists. With the same argument above, we may also define a function $H_{f^-}$ on $M$ for $f^-$. Set
$$H_f=H_{f^+}+H_{f^-},$$
then the function $H_f\in H^1(M)$ and it satisfies that
$$\begin{cases}
    \Delta H_f=0&\text{in }M;\\
    H_f=f&\text{on }\partial M,
    \end{cases}$$
and $\frac{\partial H_f}{\partial\nu}(p)\text{ exists}$ for any $q\in\partial M$. Now we may define an operator $\mathcal{D}_{M}^0$ on linear space $C^\infty_c(\partial M)$ as 
$$\mathcal{D}_M^0(f)=\frac{\partial H_f}{\partial \nu}\text{ for any }f\in C^\infty_c(\partial M).$$
 It is not hard to check that $\mathcal{D}_M^0$ is positive and symmetric on $C^\infty_c(\partial M)$.
Now we define the DtN operator for non-compact Riemannian manifold $M$ by the following theorem. 
\begin{theorem}\label{t-exten}
\cite[Theorem X.23]{RS-book} Let $A$ be a positive symmetric operator and let $q(\varphi,\psi)=(\varphi,A\psi)$ for $\varphi,\psi\in D(A)$. Then $q$ is closable quadratic form and its closure $\tilde{q}$ is the quadratic form of a unique self-adjoint operator $\tilde{A}$. $\tilde{A}$ is a positive extension of $A$, and the lower bound of its spectrum is the lower bound of $q$. 
\end{theorem}
\noindent Apply Theorem \ref{t-exten} for $A=\mathcal{D}_M^0$, $D(\mathcal{D}_M^0)=C_c^\infty(\partial M)$ and 
$$q(\varphi,\psi)=\int_{\partial M}\varphi\frac{\partial H_\psi}{\partial\nu}d\sigma\text{ for all }\varphi,\psi\in C_c^\infty(\partial M).$$
 There exists a unique self-adjoint operator $\mathcal{D}_M$ which is a positive extension of $\mathcal{D}_M^0$ and it satisfies that
\begin{align}\label{e-RQ-1}
\inf\textnormal{Spec}(\mathcal{D}_{M})=\inf\limits_{f}\frac{\int_{\partial M}H_f\cdot\frac{\partial H_f}{\partial \nu}d\sigma}{\int_{\partial M}f^2d\sigma}
\end{align}
where $f$ is taken over all functions in $C_c^\infty(\partial M)$ and $H_f$ is the unique function defined above.
For any $f\in C_c^\infty(M)$, define 
$$\textnormal{Cap}(f)=\inf\limits_{\phi}\int_M |\nabla\phi|^2dV_M,$$
where $\phi$ is taken over all functions in $C_c^\infty(M)$ such that $\phi|_{\partial M}=f$. 
Now we prove the following lemma.
\begin{lemma}\label{l-ccap}
    With the same assumptions as above, 
    $$\textnormal{Cap}(f)=\int_{\partial M}H_f\cdot\frac{\partial H_f}{\partial\nu}d\sigma=\int_{M}|\nabla H_f|^2dV_M.$$
\end{lemma}
\begin{proof}
    Notice that there exists $r_0>0$, such that for any $r>r_0$, $$\textnormal{Supp}(f)\subset \partial^E M_r.$$
For any $r_2>r_1>r_0$, $H_{f_{r_1}}$ could be regarded as function in $M_{r_2}$ by extending to $0$ in $M_{r_2}\setminus M_{r_1}$ and 
$$H_{f_{r_1}}|_{\partial M_{r_2}}=H_{f_{r_2}}|_{\partial M_{r_2}}.$$
It follows that
$$\int_{M_{r_2}}|\nabla H_{f_{r_2}}|^2dV_M\leq \int_{M_{r_1}}|\nabla H_{f_{r_1}}|^2dV_M.$$
Hence $\left\{\int_{M_{r}}|\nabla H_{f_{r}}|^2dV_M\right\}_{r>0}$ is non-increasing and  $\lim\limits_{r\to\infty}\int_{M_{r}}|\nabla H_{f_{r}}|^2dV_M$ exists.  One may check that
\begin{align}\label{e-ccap}
\textnormal{Cap}(f)=\lim\limits_{r\to\infty}\int_{M_{r}}|\nabla H_{f_{r}}|^2dV_M.
\end{align}
    For the first equality, from Green's formula and \eqref{e-ccap} we have
    \begin{align}\label{e-cpa}
       \textnormal{Cap}(f)&=\lim\limits_{r\to\infty}\int_{M_{r}}|\nabla H_{f_{r}}|^2dV_M\\
       &=\lim\limits_{r\to\infty}\int_{\partial^E M_r}f\frac{\partial H_{f_r}}{\partial\nu}d\sigma=\int_{\partial M}f\frac{\partial H_f}{\partial\nu}d\sigma,\nonumber 
    \end{align}
    where the last equality holds since $H_{f_r}$ converges to $H_f$ pointwise and $f$ is compact supported on $\partial M$. For the second equality, from \eqref{e-ccap} and Fatou's lemma, we have
    $$\int_{M_{r}}|\nabla H_{f}|^2dV_M\leq\lim\limits_{r\to\infty}\int_{M_{r}}|\nabla H_{f_{r}}|^2dV_M=\textnormal{Cap}(f).$$
    On the other hand, we have
    \begin{align}\label{e-cpa-1}
        0&\leq \int_{M_r}|\nabla(H_f-H_{f_r})|^2dV_M\\
        &=\int_{M_r}|\nabla H_f|^2dV_M+\int_{M_r}|\nabla H_{f_r}|^2dV_M-2\int_{M_r}\nabla H_f\cdot\nabla H_{f_r}dV_M.\nonumber
    \end{align}
    Together with Green's formula and \eqref{e-cpa}, it follows that
    \begin{align}\label{e-cpa-2}
        \lim\limits_{r\to\infty}\int_{M_r}\nabla H_f\cdot\nabla H_{f_r}dV_M=\lim\limits_{r\to\infty}\int_{\partial^E M_r}f\frac{\partial f_r}{\partial\nu}d\sigma=\textnormal{Cap}(f).
    \end{align}
    Combine with \eqref{e-ccap}, \eqref{e-cpa-1} and \eqref{e-cpa-2}, we have
    \begin{align*}
        \int_{M}|\nabla H_f|^2dV_M&=\lim\limits_{r\to\infty}\int_{M_r}|\nabla H_f|^2dV_M\\
        &\geq 2\lim\limits_{r\to\infty}\int_{M_r}\nabla H_f\cdot\nabla H_{f_r}dV_M-\lim\limits_{r\to\infty}\int_{M_r}|\nabla H_{f_r}|^2d\sigma\\
        &=\textnormal{Cap}(f).
    \end{align*}
    Hence $$ \int_{M}|\nabla H_f|^2dV_M=\textnormal{Cap}(f).$$ The proof is complete.
\end{proof}
\noindent Togther with \eqref{e-RQ-1} and Lemma \ref{l-ccap}, we have
\begin{align}\label{e-RQ-2}
    \inf\textnormal{Spec}(\mathcal{D}_{M})=\inf\limits_{f}\frac{\int_{M}|\nabla H_f|^2dV_{ M}}{\int_{\partial M}f^2d\sigma},
\end{align}
where $f$ is taken over all functions in $C_c^\infty(\partial M)$ and $H_f$ is the unique function defined above.

Assume $(M,g)$ is a non-compact Riemannian manifold with embedded smooth boundary. Recall that for any compact submanifold $N$ of $M$ with embedded smooth boundary, the function space $H^1_{\partial^I N}(N)$ is defined by
 $$H^1_{\partial^I N}(N)=\{f\in H^1(N);\ f=0\text{ on }\partial^I N\}.$$
 Take a point $p\in M$. For any $r>0$, denote by $$M_r=M\cap \mathbb{B}(p,r)$$
where $\mathbb{B}(p,r)$ is the geodesic ball with center $p$ and radius $r$. Notice that for any $0<r_1<r_2$ and $f\in H^1_{\partial^I M_{r_1}}(M_{r_1})$, it could be regarded as a function in $H^1_{\partial^I M_{r_2}}(M_{r_2})$ by extending to $0$ on $M_{r_2}\setminus M_{r_1}$. From \eqref{e-ray}, one may conclude that
\begin{align*}
   \xi_1\left(M_{r_2},\partial^I M_{r_2}\right)\leq \inf\limits_{f\in H^1_{\partial^I M_{r_1}}(M_{r_1})}\frac{\int_{M_{r_1}}|\nabla f|^2dV_{ M}}{\int_{\partial^E M_{r_1}}f^2d\sigma}=\xi_1\left(M_{r_1},\partial^I M_{r_1}\right).
\end{align*}
Hence $\{\xi_1(M_r,M)\}_{r>0}$ is non-increasing on $r$ and the limit $$\xi_1(M)\overset{\textnormal{def}}{=}\lim\limits_{r\to\infty}\xi_1\left(M_r,\partial^I M_r\right)$$ exists.  
 It is not hard to check that $\xi_1(M)$ does not depend on the choice of point $p$. Then we have the following lemma.
\begin{lemma}\label{l-lim-1}
With the same assumptions as above, the following equality holds,
    $$\xi_1(M)=\inf\textnormal{Spec}(\mathcal{D}_{M}).$$
\end{lemma}
\begin{proof}
    Take a sequence $\{r_i\}_{i\geq 1}$ of positive real numbers such that $$\lim\limits_{i\to\infty} r_i=\infty.$$ Let $f_{r_i}\in H^1_{\partial^I M_{r_i}}(M_{r_i})$ be the function corresponds to $\xi_1(M_{r_i},M)$, i.e.
    \begin{align}\label{l-RQ}
        \xi_1\left(M_{r_i},\partial^I M_{r_i}\right)=\frac{\int_{M}|\nabla f_{r_i}|^2dV_{M}}{\int_{\partial M}f_{r_i}^2d\sigma}
    \end{align}
    and $f_{r_i}$ is a harmonic function on $M_{r_i}$ for any $i\in\mathbb{N}^+$.
    Together with \eqref{e-RQ-2}, we have 
    $$\inf\textnormal{Spec}(\mathcal{D}_{M})\leq \frac{\int_{M}|\nabla f_{r_i}|^2dV_{M}}{\int_{\partial M}f_{r_i}^2d\sigma}=\xi_1\left(M_{r_i},\partial^I M_{r_i}\right)\text{ for any }i\in\mathbb{N}.$$
    It follows that
    $$\inf\textnormal{Spec}(\mathcal{D}_{M})\leq\lim\limits_{i\to\infty}\xi_1\left(M_{r_i},\partial^I M_{r_i}\right)=\xi_1(M).$$
    On the other hand, it follows from \eqref{e-RQ-2} that for any $\epsilon>0$, there exists a function $\phi\in C_c^\infty(\partial M)$ such that 
    \begin{align}\label{e-p-1}
    \frac{\int_{M}|\nabla H_{\phi}|^2dV_{M}}{\int_{\partial M}\phi^2d\sigma}\leq\inf\textnormal{Spec}(\mathcal{D}_{M})+\epsilon.
    \end{align}
   Since $\phi$ is compact support, there exists $r_0>0$ such that for any $r>r_0$,
   $$\textnormal{Supp}(\phi)\subset M_r.$$
   Hence $\phi\in H^1_{\partial^I M_r}(M_r)$ for all $r>r_0$. It follows that
   $$\xi_1( M)=\lim\limits_{i\to\infty}\xi_1\left(M_{r_i},\partial^I M_{r_i}\right)\leq\inf\textnormal{Spec}(\mathcal{D}_{ M})+\epsilon.$$
     The proof is complete by letting $\epsilon\to 0$.
\end{proof}

\section{Capacity}
In this section, we will introduce the definition and basic properties of capacity. It is mainly based on Chapter 2 in \cite{Ma-book}. 

Assume $(M,g)$ is a compact Riemannian manifold with smooth boundary $\partial M$, and $A,B\subset M$ are two disjoint compact subsets. Consider the following space of some smooth functions,
$$\mathfrak{R}(A,B, M)=\left\{f\in C^{\infty}\left( M\right);\ f\geq 1\text{ on }A,\ f\leq 0\text{ on }B\right\}.$$
The capacity of $A$ and $B$ relative to $ M$ is defined as 
$$\textnormal{Cap}(A,B, M)=\inf\limits_{u\in\mathfrak{R}(A,B, M)}\int\limits_{ M}|\nabla u|^2dV_{ M}.$$
\begin{lemma}\label{l-1-1}
    Assume $A,B$ are two disjoint compact subsets of $M$, then 
    $$\textnormal{Cap}(A,B, M)=\inf\limits_{u\in\mathfrak{R}^\prime(A,B, M)}\int\limits_{ M}|\nabla u|^2dV_{ M},$$
    where $\mathfrak{R}^\prime(A,B, M)$ is a space of smooth functions defined by
    $$\mathfrak{R}^\prime(A,B, M)=\left\{f\in C^{\infty}\left(M\right);\ \begin{matrix} f= 1\text{ in a neighborhood of }A,\\ \text{and }f=0\text{ in a neighborhood of }B\end{matrix}\right\}.$$
\end{lemma}
\begin{proof}
    Denote by
    $$\textnormal{Cap}^\prime(A,B, M)=\inf\limits_{u\in\mathfrak{R}^\prime(A,B, M)}\int\limits_{ M}|\nabla u|^2dV_{ M}.$$
    It is obvious that $\mathfrak{R}^\prime(A,B, M)\subset\mathfrak{R}(A,B, M)$, hence $$\textnormal{Cap}(A,B, M)\leq\textnormal{Cap}^\prime(A,B, M).$$ On the other hand, for any $\epsilon>0$, take $f\in\mathfrak{R}(A,B, M)$ such that
    $$\int_{ M}|\nabla f|^2dV_{ M}\leq \textnormal{Cap}(A,B, M)+\epsilon.$$
Take a sequence $\{\lambda_m(t)\}_{m\geq 1}$ of functions in $C^\infty(\mathbb{R})$ such that
\begin{enumerate}
    \item $0\leq\lambda_m^\prime(t)\leq 1+m^{-1}$;
    \item $\lambda_m(t)=0$ in a neighborhood of $(-\infty,0]$ and $\lambda_m(t)=1$ in a neighborhood of $[1,\infty)$;
    \item $0\leq\lambda_m(t)\leq 1$.
\end{enumerate}
Then we have 
$\lambda_m(f(x))\in\mathfrak{R}^\prime(A,B, M)$ and 
\begin{align*}\textnormal{Cap}^\prime(A,B, M)\leq\int_{ M}|\nabla \lambda_m(f)|^2dV_{ M}&\leq (1+m^{-1})^2\int_{ M}|\nabla f|^2dV_{ M}\\
&\leq (1+m^{-1})^2(\textnormal{Cap}(A,B, M)+\epsilon).
\end{align*}
By letting $m\to\infty$ and $\epsilon\to 0$, we have
$$\textnormal{Cap}^\prime(A,B, M)\leq\textnormal{Cap}(A,B, M).$$ The proof is complete.
\end{proof}

\noindent Denote by
$$\Lambda=\left\{\lambda\in C^\infty(\mathbb{R});\ \begin{matrix}
    \lambda\text{ is non-decreasing. } \lambda(t)=0\text{ for }t\leq 0,\\ \lambda(t)=1\text{ for }t\geq 1
    \text{ and }\textnormal{Supp}(\lambda)\subset (0,1).
\end{matrix}\right\}.$$
The following lemma comes from Lemma 2 in Page 144 of \cite{Ma-book}.
\begin{lemma}\label{l-fun}
    Let $g$ be a non-negative function that is integrable on $[0,1]$, then
    $$\inf\limits_{\lambda\in\Lambda}\int_0^1(\lambda^\prime)^2gdt=\left(\int_0^1\frac{dt}{g}\right)^{-1}.$$
\end{lemma}
\noindent Similar to Lemma 1 in Page 144 of \cite{Ma-book}, we give a new representation to capacity $\textnormal{Cap}(A,B, M)$ in the following lemma.
\begin{lemma}\label{l-cap}
    For any two disjoint compact subsets $A,B\subset M$, 
    $$\textnormal{Cap}(A,B, M)=\inf\limits_{u\in\mathfrak{R}(A,B, M)}\left\{\int_0^1\frac{dt}{\int_{ M_u^t}|\nabla u|ds}\right\}^{-1}$$
    where $ M_u^t=\{x\in M;\ u(x)=t\}$ for any $t\in\mathbb{R}$ and $ds$ is the $n-1$ dimensional Huasdorff measure on $ M_u^t$ induced by $dV_{ M}$.
\end{lemma}
\begin{proof}
    For $u\in\mathfrak{R}(A,B, M)$ and any $\lambda\in\Lambda$,
    \begin{align*}
         \int_{ M}|\nabla \lambda(u)|^2dV_{ M}=\int_{ M}(\lambda^\prime(u)|\nabla u|)^2dV_{ M}
         &=\int_0^1 dt\int_{ M_u^t}(\lambda^\prime(u))^2|\nabla u|ds\\
         &=\int_0^1 (\lambda^\prime(t))^2\left(\int_{ M_u^t}|\nabla u|ds\right)dt.
    \end{align*}
    Together with Lemma \ref{l-fun}, we have
    \begin{align*}\textnormal{Cap}(A,B, M)&\leq\inf\limits_{\substack{u\in\mathfrak{R}(A,B, M)\\ \lambda\in\Lambda}}\int_{ M}|\nabla \lambda(u)|^2dV_{ M} \\
    &=\inf\limits_{\substack{u\in\mathfrak{R}(A,B, M)\\ \lambda\in\Lambda}}\int_0^1 (\lambda^\prime(t))^2\left(\int_{ M_u^t}|\nabla u|ds\right)dt\\
    &\leq \inf\limits_{u\in\mathfrak{R}(A,B, M)}\left\{\int_0^1\frac{dt}{\int_{ M_u^t}|\nabla u|ds}\right\}^{-1}.
    \end{align*}
    On the other hand, for any $u\in\mathfrak{R}(A,B, M)$, we have
    \begin{align*}
        \int_{ M}|\nabla u|^2dV_{ M}=\int_0^\infty dt\int_{ M_u^t}|\nabla u|ds&\geq\int_0^1 dt\int_{ M_u^t}|\nabla u|ds\\
        &\geq\left\{\int_0^1\frac{dt}{\int_{ M_u^t}|\nabla u|ds}\right\}^{-1}.
    \end{align*}
    It follows that
    $$\textnormal{Cap}(A,B, M)\geq\inf\limits_{u\in\mathfrak{R}(A,B, M)}\left\{\int_0^1\frac{dt}{\int_{ M_u^t}|\nabla u|ds}\right\}^{-1}.$$
    The proof is complete.
\end{proof}
\noindent For any $t\in\mathbb{R}$ and $u\in C^\infty(M)$, set
$$ M_u^{\geq t}=\{x\in M;\ u(x)\geq t\}\text{ and }M_u^{\leq t}=\{x\in M;\ u(x)\leq t\}.$$
Consider the function
 $$\psi(t)=\int_0^t \frac{d\tau}{\int_{ M_u^{\geq t}}|\nabla u|ds}<\infty.$$
 Let $t(\psi)$ be the inverse function of $\psi(t)$, similar to the proof of Lemma in Page 153 of \cite{Ma-book}, one may have the following lemma.
\begin{lemma}\label{l-function}
    Assume $(M,g)$ is a Riemannian manifold with smooth boundary $\partial M$ and $u\in C^\infty\left(M\right)$ such that $$T=\sup\left\{t>0;\ \textnormal{Cap}\left( M_u^{\geq t},\  M_u^{\leq 0}\right)>0\right\}>0.$$
    With the same notations as above, then the function $t(\psi)$ is absolute continuous on segment $[0,\psi(T-\delta)]$ for any $0<\delta<T$, and 
    $$\int_{ M_u^+}|\nabla u|^2dV_{ M}\geq\int_{0}^{\psi(T-\delta)}(t^\prime(\psi))^2d\psi.$$
\end{lemma}
Now we prove the following inequality of capacity.
\begin{proposition}\label{p-cap}
    Assume $(M,g)$ is a Riemannian manifold with smooth boundary $\partial M$ and $u\in C^\infty\left(M\right)$, then
    $$\int_0^\infty \textnormal{Cap}\left( M_u^{\geq t},\  M_u^{\leq 0},\  M\right)dt\leq 4\int_{ M_u^{\geq 0}} |\nabla u|^2dV_ M,$$

\end{proposition}

\begin{proof}
 Let $T$ be the constant defined in Lemma \ref{l-function}, then
 \begin{align}\label{e-1}
     \int_0^\infty \textnormal{Cap}\left( M_u^{\geq t},\  M_u^{\leq 0},\  M\right)dt^2&=\int_0^T \textnormal{Cap}\left( M_u^{\geq t},\  M_u^{\leq 0},\  M\right)dt^2\\
     &=\int_0^{\psi(T)}\textnormal{Cap}\left( M_u^{\geq t(\psi)},\  M_u^{\leq 0},\  M\right)d(t(\psi)^2).\nonumber
 \end{align}
 For any $t>0$, setting $\xi=t^{-2}\tau^2$ and
 $$v(p)=\begin{cases}
     t^{-2}u(p)^2&\text{ if }u(p)\geq 0;\\
     -t^{-2}u(p)^2&\text{ if }u(p)\leq 0.
 \end{cases}$$
 Then we have $v\in\mathfrak{R}\left( M_u^{\geq t}, M_u^{\leq 0}, M\right)$ and
 \begin{align}\label{e-2}
     \psi(t)=\int_0^t \frac{d\tau}{\int_{ M_u^{\geq \tau}}|\nabla u|ds}
     &=\int_0^t \frac{2\tau t^{-2}d\tau}{t^{-2}\int_{ M_u^{\geq \tau}}|2u\nabla u|ds}\\
     &=\int_0^t \frac{dt^{-2}\tau^2}{\int_{ M_u^{\geq \tau}}|\nabla t^{-2}u^2|ds}\nonumber\\
     &=\int_0^1 \frac{d\xi}{\int_{v(x)=\xi}|\nabla v|ds}.\nonumber
 \end{align}
 Together with Lemma \ref{l-cap} and \eqref{e-2}, it follows that
 \begin{align*}
     \textnormal{Cap}\left( M_u^{\geq t}, M_u^{\leq 0},\  M\right)\leq\frac{1}{\psi(t)}.
 \end{align*}
 Combine with \eqref{e-1}, we have 
 \begin{align}\label{e-3}
     \int_0^\infty \textnormal{Cap}\left( M_u^{\geq t},\  M_u^{\leq 0},\  M\right)dt^2\leq 2\int_0^{\psi(T)}\frac{t(\psi)}{\psi(t)}t^\prime(\psi)d\psi.
 \end{align}
 Hardy inequality tells that
 \begin{align}\label{e-4}
     \int_0^{\psi(T)}\frac{t(\psi)^2}{\psi^2}d\psi\leq 4\int_0^{\psi(t)}t^\prime(\psi)^2d\psi.
 \end{align}
 From \eqref{e-3}, \eqref{e-4}, Lemma \ref{l-function} and Cauchy inequality, we have 
 \begin{align*}
 \left(\int_0^\infty \textnormal{Cap}\left( M_u^{\geq t},\  M_u^{\leq 0},\  M\right)dt^2\right)^2&\leq 4\left(\int_0^{\psi(T)}\frac{t(\psi)}{\psi(t)}t^\prime(\psi)d\psi\right)^2\\
 &\leq 4\int_0^{\psi(T)}\frac{t(\psi)^2}{\psi^2}d\psi\int_0^{\psi(t)}t^\prime(\psi)^2d\psi\\
 &\leq 16\left(\int_{ M_u^+}|\nabla u|^2dV_{ M}\right)^2,
 \end{align*}
 which completes the proof.
\end{proof}
\section{Proof of Main Theorems}
We will complete the proof of main theorems in this section. Assume $(M,g)$ is a compact Riemannian manifold with smooth boundary $\partial M$. Recall that for any $t>0$ and function $u\in C(M)$,
$$M^t_u=\{p\in M;\ u(p)=t\}$$
and 
$$M^{\geq t}_u=\{p\in M;\ u(p)\geq t\},\ M^{\leq t}_u=\{p\in M;\ u(p)\leq t\}.$$

\begin{theorem}[Theorem \ref{mt-1} in introduction]
    Assume $(M,g)$ is a compact Riemannian manifold with smooth boundary $\partial M$, then 
    $$\frac{1}{4}\Gamma_{\partial}( M)\leq \sigma_1( M)\leq 2\Gamma_{\partial}( M).$$
\end{theorem}
\begin{proof}
    Let $f$ be the eigenfunction corresponding to $\sigma_1( M)$. Assume $d\sigma$ is the $n-1$ dimensional Huasdorff measure on $\partial M$ induced by $dV_{ M}$ and for any measurable subset $U\subset\partial M$,
    $$m(U)=\int_{U}1d\sigma.$$
    Without losing of generality, one may assume 
    \begin{align}\label{e-assume}
    m\left(\left\{p\in\partial M; f(p)\geq 0\right\}\right)\leq m\left(\left\{p\in\partial M; f(p)\leq 0\right\}\right).
    \end{align}
    For the left inequality, from Proposition \ref{p-cap}, we have
    \begin{align}\label{in-1}
        \sigma_1( M)\int_{\partial M}f^2d\sigma&=\int_{ M}|\nabla f|^2dV_{ M}\\
        &\geq \frac{1}{2}\int_0^\infty t\cdot\textnormal{Cap}\left( M_f^{\geq t}, M_f^{\leq 0}, M\right)dt.\nonumber
      \end{align}
    From the monotonicity of capacity, we have
    $$\textnormal{Cap}\left( M_f^{\geq t}, M_f^{\leq 0}, M\right)\geq \textnormal{Cap}\left( M_f^{\geq t}\cap\partial M, M_f^{\leq 0}\cap\partial M, M\right),$$
    which together with assumption \eqref{e-assume}, \eqref{in-1} and the definition of $\Gamma_{\partial}(M)$ implies that 
        \begin{align*}
        \sigma_1( M)\int_{\partial M}f^2d\sigma&\geq \frac{1}{2}\int_0^\infty t\cdot\textnormal{Cap}\left( M_f^{\geq t}\cap\partial M, M_f^{\leq 0}\cap\partial M, M\right)dt\\
        &\geq\frac{1}{2}\Gamma_{\partial}( M)\int_0^\infty t\cdot m(\{p\in\partial M;\ f(p)\geq t\})dt\\
        &=\frac{1}{4}\Gamma_{\partial}( M)\int_{\partial M}f^2d\sigma.
    \end{align*}
    It follows that
    $$\sigma_1( M)\geq\frac{1}{4}\Gamma_{\partial}( M).$$
    Now we consider the right inequality, for any $\epsilon>0$, there exist two disjoint compact sets $A,B\subset \partial M$ such that
   \begin{align}\label{e-2-1}
       \frac{\textnormal{Cap}(A,B, M)}{\min\{m(A),m(B)\}}<\Gamma_{\partial}( M)+\epsilon.
   \end{align}
    There also exists a function $g\in C^\infty(M)$ such that
   \begin{align}\label{e-2-2}
       g\equiv 1\text{ on }A,\ g\equiv 0\text{ on }B
   \end{align}
   and 
   \begin{align}\label{e-2-3}
       \int_{ M}|\nabla g|^2dV_{ M}<\textnormal{Cap}(A,B, M)+\epsilon\cdot\min\{m(A),m(B)\}.
   \end{align}
   Take $$c=\frac{1}{m(\partial M)}\int_{\partial M} gd\sigma,$$
   then we have
   $$\int_{\partial M}(g-c)d\sigma=0.$$
   It follows from \eqref{e-2-1}, \eqref{e-2-2} and \eqref{e-2-3} that
   \begin{align*}
       \sigma_1( M)&\leq\frac{\int_{ M}|\nabla (g-c)|^2dV_{ M}}{\int_{\partial M}(g-c)^2d\sigma}\\
       &\leq\frac{\textnormal{Cap}(A,B, M)+\epsilon\cdot\min\{m(A),m(B)\}}{(1-c)^2\cdot m(A)+c^2\cdot m(B)}\\
       &\leq 2\cdot\frac{\textnormal{Cap}(A,B, M)+\epsilon\cdot\min\{m(A),m(B)\}}{\min\{m(A),m(B)\}}\\
       &\leq 2\Gamma_{\partial}( M)+4\epsilon.
   \end{align*}
   It follows that
   $$\sigma_1( M)\leq 2\Gamma_{\partial}( M)$$
   by letting $\epsilon$ tend to $0$. The proof is complete.
\end{proof}
\begin{rem*}
    Consider the Steklov-Dirichlet problem \eqref{s-d-p}, set
    $$\Gamma_Y(X)=\inf\limits_{\substack{F\subset\partial X\setminus Y\\
    \text{compact}}}\frac{\textnormal{Cap}(Y,F,X)}{m(F)}.$$
    Similar to the proof of Theorem \ref{mt-1}, one may obtain that
\begin{align}\label{e-ineq}
    \frac{1}{4}\Gamma_Y(X)\leq \xi_1(X,Y)\leq 2\Gamma_Y(X).
\end{align}
\end{rem*}

Now we turn to the proof of Theorem \ref{mt-2}. Assume $ (M,g)$ is a Riemannian manifold with smooth boundary $\partial M$. Recall that for any compact subset $F\subset\partial M$,
$$\textnormal{Cap}(F, M)=\inf\limits_{f}\int\limits_{ M}|\nabla f|^2dV_{ M},$$
where $f$ is taken over all smooth functions such that 

    $$f\equiv 0\text{ on }F\text{ and }
    f\in C_c^\infty(M).$$
    Moreover
    $$\Gamma_{\partial}( M)=\inf\limits_{\substack{F\subset\partial M\\ \text{compact}}}\frac{\textnormal{Cap}(F, M)}{m(F)}.$$
For compact submanifold $N$ of $ M$ with embedded boundary, define $\Gamma_{\partial}(N, M)$ by
$$\Gamma_{\partial}(N, M)=\Gamma_{\partial^I N}(N)=\inf\limits_{F\in \partial^E N}\frac{\textnormal{Cap}(F,\partial^I N,N)}{m(F)}.$$
Then one may have the following lemma.
\begin{lemma}\label{l-lim-2}
Let $(M,g)$ be a non-compact Riemannian manifold with embedded smooth boundary and $p$ is a point in $ M$. For any $r>0$, assume $\mathbb{B}(p,r)$ is the geodesic ball with center $p$ and radius $r$. Denote by $$ M_r= M\cap\mathbb{B}(p,r).$$
Then the following equality holds,
$$\Gamma_{\partial}( M)=\lim\limits_{r\to\infty}\Gamma_{\partial}( M_r, M).$$
\end{lemma}
\begin{proof}
    We firstly recall that for any two disjoint compact subsets $A,B$ of $ M$, the set $\mathfrak{R}^\prime(A,B, M)$ of smooth functions is defined as
    $$\mathfrak{R}^\prime(A,B, M)=\left\{f\in C^{\infty}\left(M\right);\ \begin{matrix} f= 1\text{ in a neighborhood of }A,\\ \text{and }f=0\text{ in a neighborhood of }B\end{matrix}\right\}.$$
     Assume $r>s>0$ and $F\subset \partial ^E M_s\subset \partial^E M_r$. For any function $f\in\mathcal{R}^\prime(F,\partial^I M_s, M_s)$, it could be regarded as a function in $\mathcal{R}^\prime(F,\partial^I M_r, M_r)$ by extending to $0$ on $ M_r\setminus M_s$. Then we have
    \begin{align*}
        \textnormal{Cap}(F,\partial^I M_r, M_r)&\leq\inf\limits_{f\in\mathfrak{R}^\prime(F,\partial^I M_s, M_s)}\int_{ M_s}|\nabla f|^2dV_{ M}\\
        &=\textnormal{Cap}(F,\partial^I M_s, M_s).
    \end{align*}
    Hence 
    $$\frac{\textnormal{Cap}(F,\partial^I M_r, M_r)}{m(F)}\leq\frac{\textnormal{Cap}(F,\partial^I M_s, M_s)}{m(F)},$$
    which implies that 
    $$\Gamma_{\partial}( M_r, M)\leq\inf\limits_{\substack{F\subset \partial^E  M_s\\ \text{compact}}}\frac{\textnormal{Cap}(F,\partial^I  M_s, M_s)}{m(F)}=\Gamma_{\partial}( M_s, M).$$
    Hence $\Gamma_{\partial}( M_r, M)$ is non-increasing in $r$ and $\lim\limits_{r\to\infty}\Gamma_{\partial}( M_r, M)$ exists. 
    For any fixed $r>0$ and compact subset $F\subset \partial^E M_r$, $F$ is also a compact subset of $\partial M$, from the definition of $\Gamma_\partial( M)$ and $\Gamma_{\partial}( M_r, M)$, one may obtain that 
    \begin{align*}
    \Gamma_{\partial}( M)&=\inf\limits_{\substack{F\subset \partial^E  M_r\\ \text{compact}}}\frac{\textnormal{Cap}(F, M)}{m(F)}\\
    &\leq\inf\limits_{\substack{F\subset \partial^E  M_r\\ \text{compact}}}\frac{\textnormal{Cap}(F, M_r)}{m(F)}=\Gamma_{\partial}( M_r, M).
    \end{align*}
    Hence 
    $$\Gamma_{\partial}( M)\leq\lim\limits_{r\to\infty}\Gamma_{\partial}( M_r, M).$$
    On the other hand, for any $\epsilon>0$, there exists a compact set $F_0\subset\partial M$ such that
    $$\frac{\textnormal{Cap}(F_0, M)}{m(F_0)}<\Gamma_{\partial}( M)+\epsilon.$$
    Take a function $f: M\to\mathbb{R}$ such that
    \begin{enumerate}
        \item $f\equiv 1\text{ on }F_0$ and $f\in C_c^\infty(M)$;
        \item $\int_{ M}|\nabla f|^2dV_{ M}<\textnormal{Cap}(F_0, M)+\epsilon\cdot m(F_0)$. 
    \end{enumerate}
   Take $r>0$ such that $\textnormal{supp}(f)\subset M_r$, then $f$ could be regarded as a function in $\mathfrak{R}(F_0,\partial^IM_r,M_r)$. Therefore
   \begin{align*}
       \Gamma_{\partial}(M_r,M)&\leq \frac{\textnormal{Cap}(F_0,\partial^I  M_r, M_r)}{m(F_0)}\\
       &\leq \frac{\int_M|\nabla f|^2dV_M}{m(F_0)}\\
       &\leq \frac{\textnormal{Cap}(F_0,M)}{m(F_0)}+\epsilon<\Gamma_{\partial}(M)+2\epsilon.
   \end{align*}
  It follows that
   $$\lim\limits_{r\to\infty}\Gamma_{\partial}( M_r, M)\leq \Gamma_{\partial}( M)$$
   by letting $\epsilon\to 0$. The proof is complete.
\end{proof}

\begin{proof}[Proof of Theorem \ref{mt-2}]
 For any $r>0$, apply \eqref{e-ineq} to $X=M_r$ and $Y=\partial^I M_r$, we have
\begin{align}\label{e-com}
\frac{1}{4}\Gamma_{\partial}(M_r, M)\leq \xi_1\left(M_r,\partial^I M_r\right)\leq 2\Gamma_{\partial}(M_r, M).
\end{align}
   Together with Lemma \ref{l-lim-1}, Lemma \ref{l-lim-2} and \eqref{e-com}, the proof is complete.
\end{proof}
\section{Estimates of Steklov eigenvalues}
\subsection{Compact hyperbolic surface with boundaries}
As a direct application, we estimate the first Steklov eigenvalue for hyperbolic surface $S_{g,n}$ with genus $g$ and $n$ geodesic boundaries, where $2g-2+n\geq 1$. Assume the shortest geodesic boundary $\gamma$ of $S_{g,n}$ has length $l_0$, now we estimate the first Steklov eigenvalue $\sigma_1\left(S_{g,n}\right)$ by $l_0$. Take $\rho_0>0$ such that
$$\sinh\frac{l_0}{2}\sinh \rho_0=1.$$
Consider the set
$$\omega_{\rho_0}(\gamma)=\{p\in S_{g,n};\ \textnormal{dist}(p,\gamma)\leq \rho_0\},$$
from Collar's lemma (one may refer to Theorem 4.1.1 in \cite{Bu-book}), $\omega_{\rho_0}(\gamma)$ is a topological cylinder (see Figure \ref{fig:02}), and the hyperbolic metric could be represented as 
$$ds^2=d\rho^2+l_0^2\cosh^2\rho dt^2.$$
 \begin{figure}[ht]
\centering
\includegraphics[width=\textwidth]{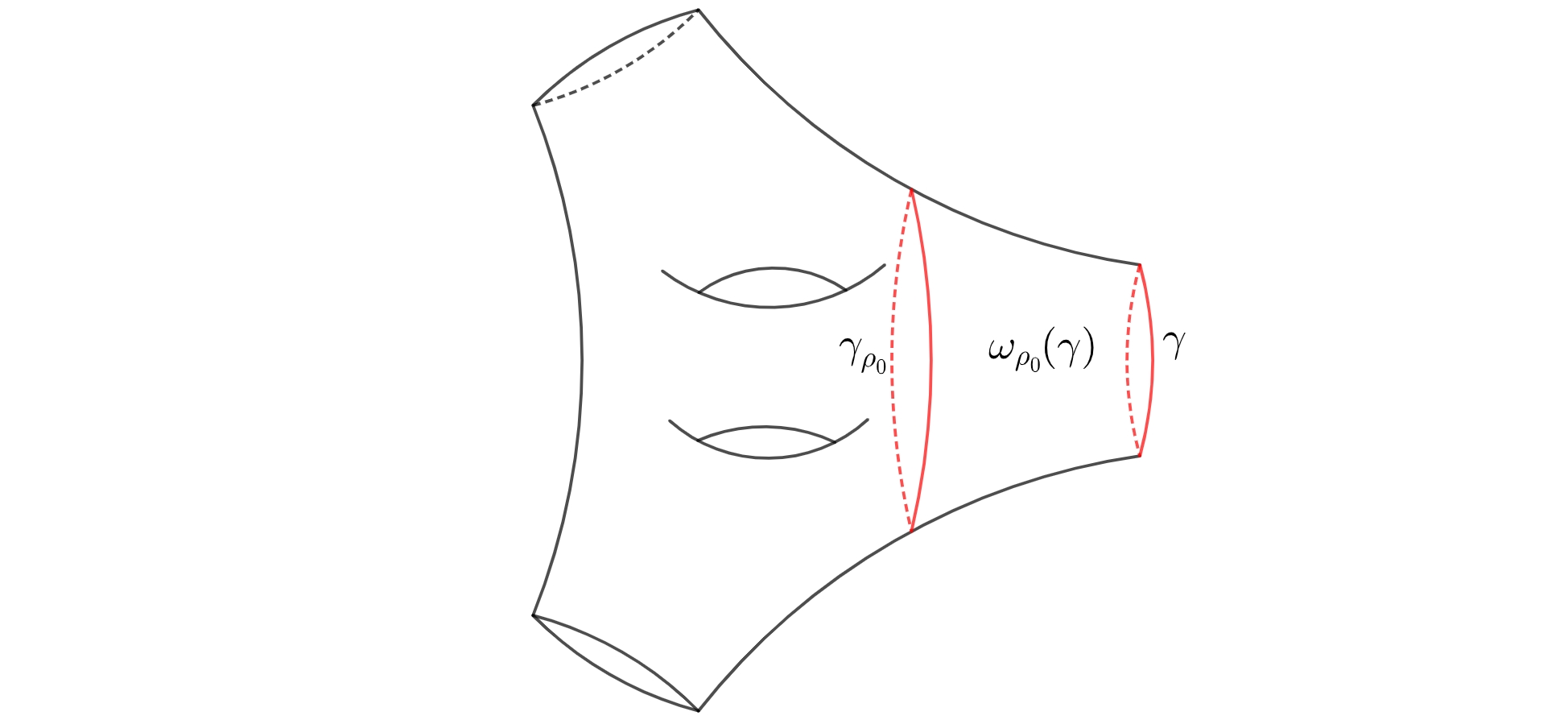}
\caption{}
\label{fig:02}
\end{figure}

\noindent Assume $\rho_1\leq \rho_0$ is a fixed positive real number, consider functions
$$f(u)=\frac{1}{\cosh u},\ h(t)=\sin2\pi t$$
and 
$$g(\rho)=\begin{cases}
    1-\frac{\rho}{\rho_1} & \text{if } 0\leq \rho\leq \rho_1;\\
    0 & \text{if }\rho_1\leq \rho\leq \rho_0.
\end{cases}$$
Now we define functions $\Phi,\Psi: \omega_{\rho_0}(\gamma)\to\mathbb{R}$ as follows. For any point $p\in\omega_{\rho_0}(\gamma)$, it could be uniquely represented as $p=(\rho,t)$, where $0\leq \rho\leq \rho_0$ and $0\leq t\leq 1$. Take
$$\Phi(p)=F(\rho)=
    1-
\frac{\int_0^{\rho}f(u)du}{\int_0^{\rho_0}f(u)du}$$
and
$$\Psi(p)=g(\rho)\cdot h(t).$$
It is not hard to check that
$$\nabla \Phi=(F^\prime(\rho),0)\text{ and }\nabla\Psi=\left(g^\prime(\rho)h(t),\frac{g(\rho)h^\prime(t)}{l_0\cosh \rho}\right).$$
Then we have
\begin{align}\label{int-grad}
\int_{\omega_{\rho_0}
}\nabla\Phi\cdot\nabla\Psi ds^2=\int_0^{\rho_0}F^\prime(\rho)g^\prime(\rho)\cosh \rho d\rho\cdot\int_0^{1}h(t)dt=0.
\end{align}
Now we estimate the square integrations of $|\nabla \Phi|$ and $|\nabla\Psi|$.
\begin{align}\label{int-phi}
    \int_{\omega_{\rho_0}(\gamma)}|\nabla\Phi|^2ds^2&=\int_0^{\rho_0}|F^\prime(\rho)|^2\cdot l_0\cosh \rho d\rho\\
    &=l_0\cdot\frac{\int_0^{\rho_0}|f(\rho)|^2\cosh \rho d\rho}{\left(\int_0^{\rho_0}f(t)dt\right)^2}\nonumber\\
    &= \frac{l_0}{\int_0^{\rho_0}\frac{1}{\cosh} \rho d\rho},\nonumber
\end{align}
\begin{align}\label{int-psi}
    \int_{\omega_{\rho_0}(\gamma)}|\nabla\Psi|^2ds^2&=\int_0^{1}\int_0^{\rho_0}|\nabla\Psi|^2\cdot l_0\cosh\rho d\rho dt\\
    &=\frac{l_0}{2}\int_0^{\rho_0}|g^\prime(\rho)|^2\cosh\rho d\rho+\frac{1}{2l_0}\int_0^{\rho_0}\frac{|g(\rho)|^2}{\cosh\rho}d\rho\nonumber\\
    &\leq\frac{l_0}{2}\cdot\frac{\sinh\rho_1}{\rho_1^2}+\frac{\rho_1}{6l_0}.\nonumber
\end{align}
Direct calculation implies that
\begin{align}\label{int-boundary}
    \int_{\gamma}|\Phi|^2=l_0\text{ and }\int_{\gamma}|\Psi|^2=l_0\int_0^1 |h(t)|^2dt=\frac{l_0}{2}.
\end{align}

Case I: $n=1$. Consider the Steklov-Dirichlet problem \eqref{s-d-p} for the case that $X=\omega_{\rho_0}(\gamma)$ and $Y=\gamma_{\rho_0}$.  We firstly estimate $\xi_2(\omega_{\rho_0}(\gamma),\gamma_{\rho_0})$. Consider
$$\mathcal{F}=\{a\Phi+b\Psi;\ a,b\in\mathbb{R}\}.$$
Then $\mathcal{F}$ is a $2$-dimension subspace of $$H^1_{\gamma_{\rho_0}}(\omega_{\rho_0}(\gamma))=\{f\in H^1(\omega_{\rho_0}(\gamma));\ f\equiv 0\text{ on }\gamma_{\rho_0}\}.$$
And for any $\varphi=a\Phi+b\Psi\in\mathcal{F}$, together with \eqref{int-grad}, \eqref{int-phi}, \eqref{int-psi} and \eqref{int-boundary}, we have
\begin{align*}
    \frac{\int_{\omega_{\rho_0}(\gamma)}|\nabla\varphi|^2}{\int_{\gamma}|\varphi|^2}&=\frac{a^2\int_{\omega_{\rho_0}(\gamma)}|\nabla\Phi|^2+b^2\int_{\omega_{\rho_0}(\gamma)}|\nabla\Psi|^2}{a^2\int_{\gamma}|\Phi|^2+b^2\int_{\gamma}|\Psi|^2}\\
    &\leq\max\left\{\frac{\int_{\omega_{\rho_0}(\gamma)}|\nabla\Phi|^2}{\int_{\gamma}|\Phi|^2},\frac{\int_{\omega_{\rho_0}(\gamma)}|\nabla\Psi|^2}{\int_{\gamma}|\Psi|^2}\right\}\\
    &\leq\max\left\{\frac{1}{\int_0^{\rho_0}\frac{1}{\cosh\rho}  d\rho},\frac{\sinh\rho_1}{\rho_1^2}+\frac{\rho_1}{3l_0^2}\right\}.
\end{align*}
Together with $(2.10)$ in \cite{ste-sur}, it follows that
$$\sigma_1(S_{g,1})\leq \xi_2(\omega_{\rho_0}(\gamma),\gamma_{\rho_0})\leq \max\left\{\frac{1}{\int_0^{\rho_0}\frac{1}{\cosh\rho} d\rho},\frac{\sinh\rho_1}{\rho_1^2}+\frac{\rho_1}{3l_0^2}\right\}. $$

Case II: $n\geq 2$. Define function $\Phi_0:S_{g,n}\to\mathbb{R}$ as follows,
$$\Phi_0(p)=\begin{cases}
    \Phi(p)& \text{if }p\in\omega_{\rho_0};\\
    0 & \text{otherwise}.
\end{cases}$$
From Theorem \ref{mt-1}, we have
$$\sigma_1\left(S_{g,n}\right)\leq 2\Gamma_{\partial}\left(S_{g,n}\right)\leq 2\frac{\int_{S_{g,n}}|\nabla \Phi_0|^2}{l_0}\leq \frac{2}{\int_0^{\rho_0}\frac{1}{\cosh\rho} d\rho}.$$
For small $l_0$, we have $l_0<1<\rho_0$, take $\rho_1=l_0$. Then 
$$\frac{1}{\int_0^{\rho_0}\frac{1}{\cosh\rho} d\rho}<\frac{e+e^{-1}}{2}$$
and
$$\frac{\sinh\rho_1}{\rho_1^2}+\frac{\rho_1}{3l_0^2}\sim\frac{4}{3l_0}\text{ as }l_0\to 0.$$
One may obtain the following corollaries.
\begin{enumerate}
    \item Assume $\{S_{g_i,1}\}_{i\geq 0}$ is a sequence of compact hyperbolic surfaces with genus $g_i$ and a geodesic boundary $\gamma_i$. If 
    $$\lim\limits_{i\to\infty}\ell(\gamma_i)=0$$
    then the normalized Steklov eigenvalue $\{\sigma_1(S_{g_i,1})\ell(\gamma_i)\}$ satisfying that  
    $$\limsup\limits_{i\to\infty}\sigma_1(S_{g_i,1})\ell(\gamma_i)\leq\frac{4}{3}.$$
    \item Assume $S_{g,n}$ is a compact hyperbolic surfaces with $n$ geodesic boundaries. If a boundary curve of $S_{g,n}$ has length $l_0<1<\rho_0$, then
    $$\sigma_1(S_{g,n})\leq e+e^{-1}\sim 3.086.$$
\end{enumerate}
\begin{rem*}
   It has been proved in \cite[Theorem A1]{Ko14} that for any compact surface $S$ with genus $0$, 
   $$\sigma_1(S)\ell(\partial S)<8\pi.$$
   According to the calculations above, for the case that $S=S_{g,n}(n\geq 2)$ is a compact hyperbolic surface with totally boundary length small, one may improve $8\pi$ to a smaller positive constant. 
  
\end{rem*}

\noindent Also consider hyperbolic surface $S_{g,n}$ with genus $g$ and $n$ boundary components. Assume each boundary component is an equidistance curve for some simple closed geodesic. For example, in the first picture of Figure \ref{fig:03}, $\gamma$ is an equidistance curve for the simple closed geodesic $\alpha$.
\begin{figure}[ht]
\centering
\includegraphics[width=0.8\textwidth]{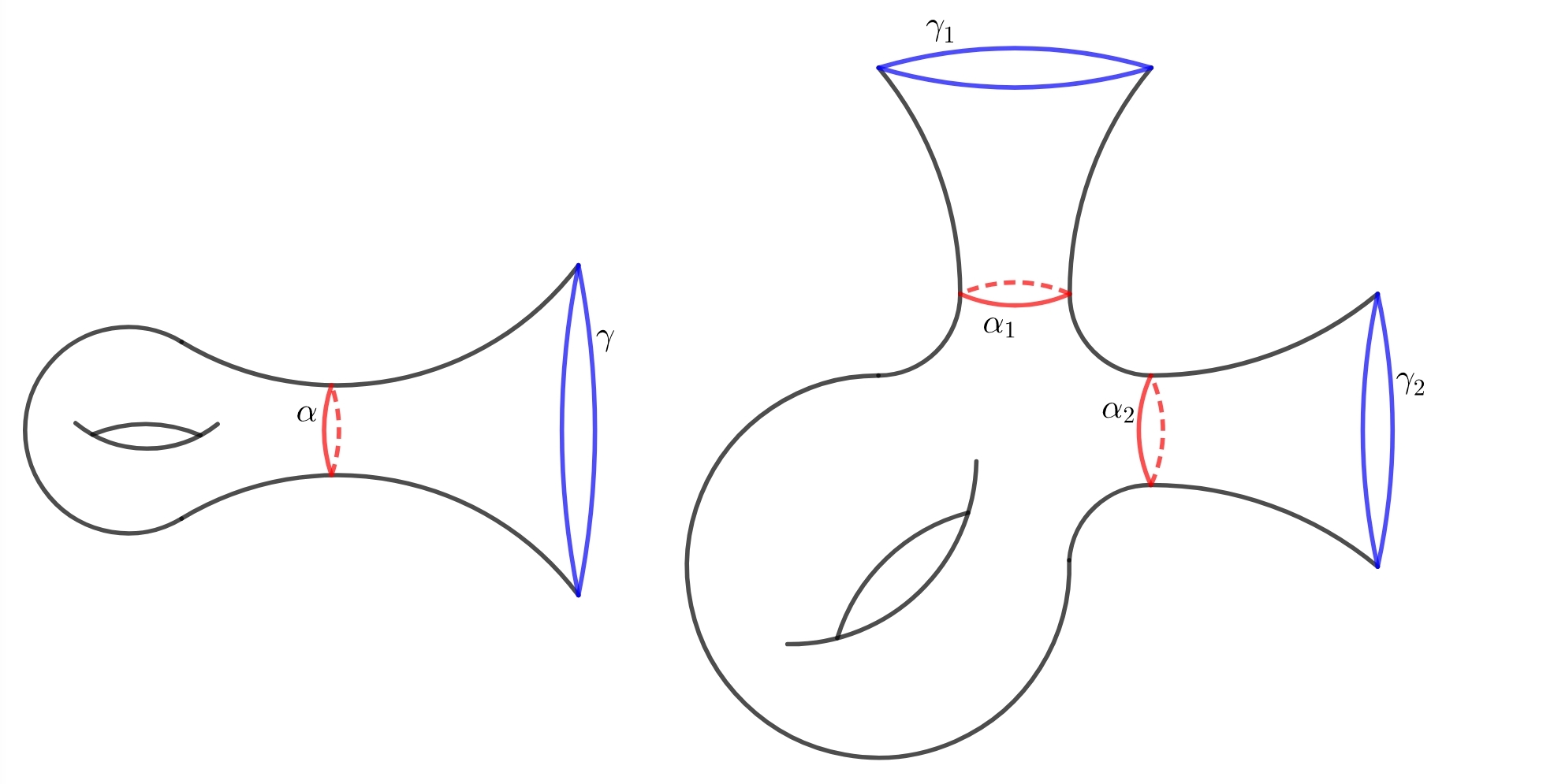}
\caption{}
\label{fig:03}
\end{figure}

Case I: $n=1$. As showed in \cite{ste-sur} (Example 4.5 in Page 39), the first Steklov eigenvalue $\sigma_1(S_{g,1})$ has a universal lower bound as boundary length tends to infinity.

Case II: $n\geq 2$. Let $\gamma_1$ and $\gamma_2$ be two boundary curves corresponding to closed geodesics $\alpha_1$ and $\alpha_2$ respectively (see the second picture in Figure \ref{fig:03}). Assume $\ell(\gamma_1)\leq\ell(\gamma_2)$, with the same method as above, one may obtain that there exists a universal constant $C$ such that
$$\sigma_1(S_{g,n})\leq \frac{C\ell(\alpha_1)}{\ell(\gamma_1)}.$$
Hence the first Steklov eigenvalue $\sigma_1(S_{g,n})$ tends to $0$ as boundary length $\ell(\gamma_1)$ tends to infinity.

\subsection{Non-compact Riemannian manifold}
Now we consider the case of non-compact Riemannian manifold, with the same method in \cite[Lemma 6.5]{Pe-book}, we prove the following lemma.
\begin{lemma}\label{l-ste}
    Let $(M,g)$ be a non-compact Riemannian manifold with embedded smooth boundary $\partial M$, if there exists a positive function $f$ on $M$ and $\lambda>0$ such that
    $$\begin{cases}
        \Delta f=0 & \text{in }M;\\
        \frac{\partial f}{\partial\nu}=\lambda f& \text{on }\partial M,
    \end{cases}$$
where $\frac{\partial}{\partial\nu}$ is the normal outward derivative along $\partial M$. Then 
$$\inf\textnormal{Spec}(\mathcal{D}_M)\geq\lambda.$$
\end{lemma}
\begin{proof}
    Take a point $p\in M$. For any $r>0$, denote by
    $$M_r=M\cap\mathbb{B}(p,r),$$
    where $\mathbb{B}(p,r)$ is the geodesic ball with center $p$ and radius $r$. 
    
    Consider the Steklov-Dirichlet problem \eqref{s-d-p} for the case that $X=M_r$ and $Y=\partial^I M_r$. Assume it has the first eigenvalue $\xi_{1,r}=\xi_1(M_r,\partial^I M_r)$ and $f_r$ is the corresponding eigenfunction. Then $f_r$ is positive in the interior of $M_r$ and 
    \begin{align*}
        (\xi_{1,r}-\lambda)\int_{\partial^E M_r}ff_r&=\int_{\partial^E M_r}f\frac{\partial f_r}{\partial\nu}-\int_{\partial^E M_r}f_r\frac{\partial f}{\partial\nu}\\
        &=-\int_{\partial^I M_r}f\frac{\partial f_r}{\partial\nu}\geq 0,
    \end{align*}
    where the last inequality holds since $f_r$ is positive in the interior of $M_r$ and $f\equiv 0$ on $\partial^I M_r$ imply that
   $$\frac{\partial f_r}{\partial\nu}\leq 0\text{ on }\partial^I M_r.$$ 
   Together with Lemma \ref{l-lim-1}, we have
   $$\inf\textnormal{Spec}(\mathcal{D}_{M})=\lim\limits_{r\to \infty}\xi_{1.r}\geq\lambda.$$
\end{proof}
\noindent As a direct application of the above lemma, we consider 
$$\mathbb{D}^{n+1}_+=\left\{(x_1,...,x_{n},y);\ \sum\limits_{i=1}^n x_i^2<1,\ y\geq 0\right\}$$
endowed with the standard hyperbolic metric 
$$ds^2=4\cdot\frac{\sum\limits_{i=1}^{n} dx_i^2+dy^2}{\left(1-\sum\limits_{i=1}^{n} x_i^2-y^2\right)^2}.$$
Then $\mathbb{D}^{n+1}_+$ is a non-compact Riemannian manifold with boundary
$$\partial\mathbb{D}^{n+1}_+=\left\{(x_1,...,x_{n},0);\ \sum\limits_{i=1}^{n}x_i^2<1\right\}.$$
One easily checks that $\mathbb{D}^{n+1}_+$ is isometric to 
$$\mathbb{H}^{n+1}_R=\left\{(x_1,...,x_n,y);\ y>0,\ x_n\geq 0\right\}$$
endowed with hyperbolic metric 
$$ds^2=\frac{1}{y^2}\left(\sum\limits_{i=1}^{n}dx^2_i+dy^2\right)$$
and boundary
$$\partial\mathbb{H}^{n+1}_R=\{(x_1,...,x_{n-1},0,y);\ y>0\}.$$
For any $p=(x_1,...,x_n,y)\in\mathbb{H}^{n+1}_R$, assume $p^\prime$ is the hyperbolic projection of $p$ on $\partial\mathbb{H}^{n+1}_R$ (see Figure \ref{fig:04}), then $p^\prime$ could be represented as 
$$p^\prime=\left(x_1,...,x_{n-1},0,\sqrt{x_n^2+y^2}\right)$$
and 
$$\cosh\textnormal{dist}(p,\partial\mathbb{H}^{n+1}_R)=\cosh\textnormal{dist}(p,p^\prime)=\frac{\sqrt{x_n^2+y^2}}{y}.$$
\begin{figure}[ht]
\centering
\includegraphics[width=0.8\textwidth]{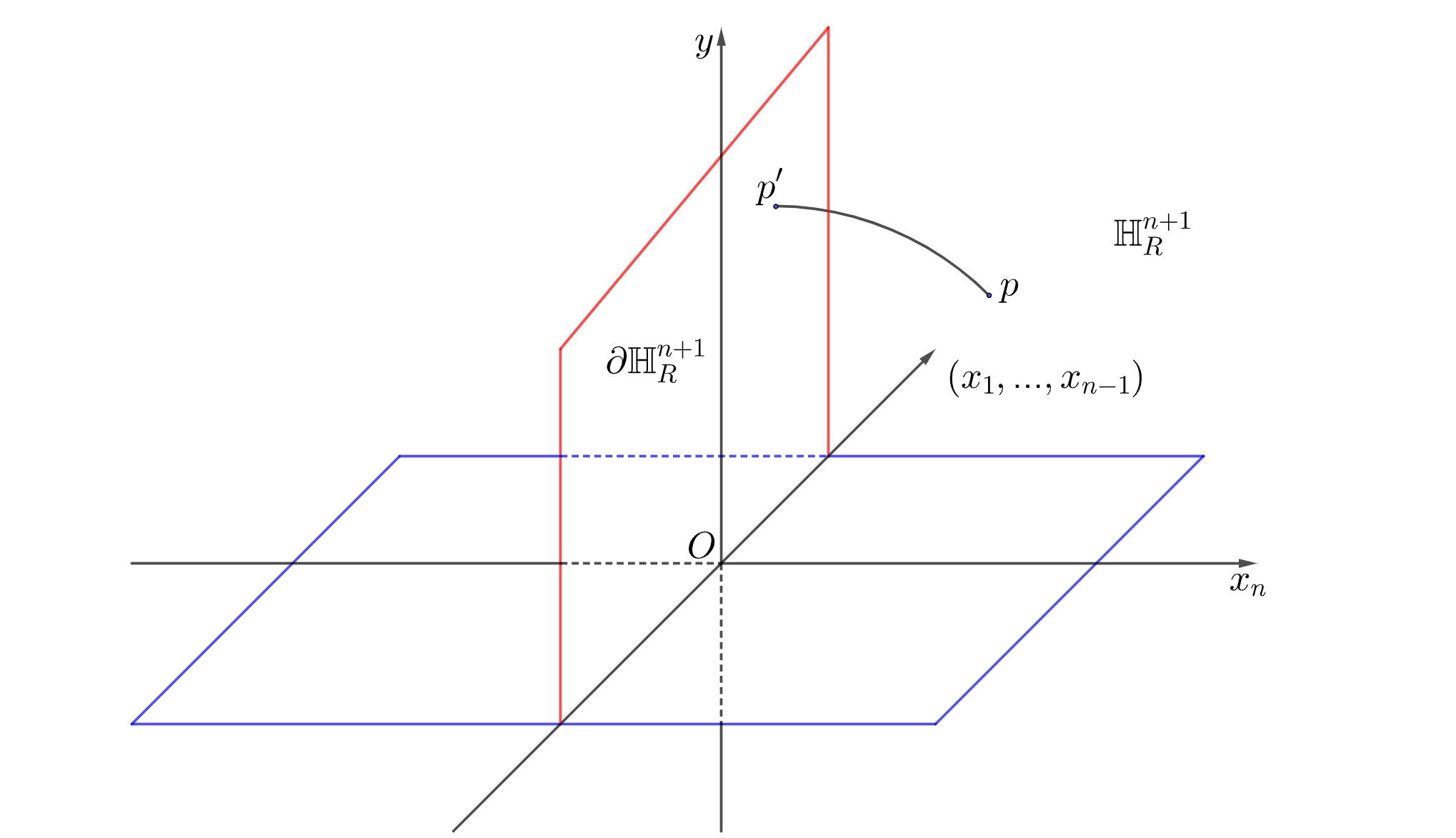}
\caption{}
\label{fig:04}
\end{figure}

\noindent Hence $\mathbb{H}^{n+1}_R$ has another coordinate representation $(u_1,...,u_n,r)$ such that for any point $p=(x_1,...,x_n,y)\in\mathbb{H}^{n+1}_R$,
$$\begin{cases}
    u_i(p)=x_i \text{ for } 1\leq i\leq n-1;\\
    u_n(p)=y;\\
    r(p)=\textnormal{dist}(p,\partial\mathbb{H}^{n+1}_R)=\textnormal{arccosh} \frac{\sqrt{x_n^2+y^2}}{y}.
\end{cases}$$
Direct calculation implies that under this coordinate, hyperbolic metric could be represented as 
$$ds^2=dr^2+\frac{\cosh^2 r}{u^2_n}\sum\limits_{i=1}^ndu_i^2.$$
Moreover the Laplacian operator could be represented as
$$\Delta_{\mathbb{H}^{n+1}_R}=\frac{\partial^2}{\partial r^2}+\frac{n\sinh r}{\cosh r}\cdot\frac{\partial}{\partial r}+\frac{u_n^2}{\cosh^2 r}\sum\limits_{i=1}^n \frac{\partial^2}{\partial u_i^2}.$$
Define a function $\varphi_n:\mathbb{H}^{n+1}_R\to\mathbb{R}$ as follows: for any $p=(u_1,...,u_n,r)$,
$$\varphi_n(p)=1-c_n\int_0^r\frac{1}{\cosh^n s}ds,$$
where $c_n=\left(\int_0^\infty\frac{1}{\cosh^n s}ds\right)^{-1}$. 
Then one easily checks that 
\begin{enumerate}[(a)]
    \item $\Delta_{\mathbb{H}^{n+1}_R}\varphi_n=0$ on $\mathbb{H}^{n+1}_R$;
    \item $\varphi_n(p)=1$ for $p\in\partial\mathbb{H}^{n+1}_R$ and $\varphi_n(p)\geq 0$ for $p\in\mathbb{H}^{n+1}_R$;
    \item $\frac{\partial \varphi_n}{\partial\nu}(p)=\frac{\partial \varphi_n}{\partial r}(p)=c_n\varphi_n(p)$ for $p\in\partial\mathbb{H}^{n+1}_R$,
\end{enumerate}
where $\frac{\partial}{\partial\nu}$ is the outward normal derivative along $\partial\mathbb{H}^{n+1}_R$. From Lemma \ref{l-ste}, we have 
$$\inf\textnormal{Spec}\left(\mathcal{D}_{\mathbb{D}^{n+1}_R}\right)=\inf\textnormal{Spec}\left(\mathcal{D}_{\mathbb{H}^{n+1}_R}\right)\geq c_n=\left(\int_0^\infty\frac{1}{\cosh^n s}ds\right)^{-1}.$$

Now we give the an expression of the DtN operator $\mathcal{D}_{\mathbb{D}^{n+1}_+}$ for $n\geq 1$. The Green function on $\mathbb{D}^{n+1}$ could be represented as 
$$G(p,q)=\frac{1}{\textnormal{Vol}(S^{n})}\int_{\textnormal{dist}(p,q)}^\infty \frac{1}{\sinh^{n}t}dt\text{ for }p,q\in\mathbb{D}^{n+1},$$
where $\textnormal{Vol}(S^n)$ is Euclidean volume of $n-$dimension sphere.
 The Green function $G^+$ on $\mathbb{D}^{n+1}_+$ could be represented as 
$$G^+(p,q)=G(p,q)-G(p,q^\prime),$$
 where $p,q\in\mathbb{D}^{n+1}_+$ and $q^\prime$ is the mirror symmetric point of $q$ corresponding to $\partial\mathbb{D}^{n+1}_+$.
 Then for any point $p_0\in\mathbb{D}^{n+1}_+$ and $q_0\in\partial\mathbb{D}^{n+1}_+$, 
 \begin{align*}
     P(p_0,q_0)&=-\frac{\partial G^+(p_0,q)}{\partial\nu}|_{q=q_0}\\
     &=-\frac{2}{\textnormal{Vol}(S^n)}\cdot\frac{1}{\sinh^{n}\textnormal{dist}(p_0,q_0)}\cdot\frac{\partial\ \textnormal{dist}(p_0,q)}{\partial\nu}|_{q=q_0},
 \end{align*}
 where $\frac{\partial}{\partial\nu}$ is the outward normal derivative along $\partial\mathbb{D}^{n+1}_R$. For any function $f$, the function $H_f$ on $\partial\mathbb{D}^{n+1}_+$ such that
$$\begin{cases}
 \Delta_{\mathbb{D}^{n+1}_+}H_f=0\text{ on }\mathbb{D}^{n+1}_+;\\
 H_f\equiv f\text{ on }\partial\mathbb{D}^{n+1}_+,
\end{cases}$$
could be represented as 
\begin{align}\label{har-ext}
H_f(p)=\int_{\partial\mathbb{D}^{n+1}_+}P(p,q)f(q)d\sigma_q\text{ for any }p\in\mathbb{D}^{n+1}_+.
\end{align}
For any $p_0\in\partial\mathbb{D}^{n+1}_+$,
\begin{align*}
   -\frac{\partial H_f(p_0)}{\partial\nu}&=\lim\limits_{p\to p_0}\frac{H_f(p)-f(p_0)}{\textnormal{dist}(p,p_0)}\\
   &=\lim\limits_{p\to p_0}\frac{1}{\textnormal{dist}(p,p_0)}\left(\int_{\partial\mathbb{D}^n_+}P(p,q)(f(q)-f(p_0))d\sigma_q\right)\\
   &+\lim\limits_{p\to p_0}\frac{f(p_0)}{\textnormal{dist}(p,p_0)}\left(\int_{\partial\mathbb{D}^n_+}P(p,q)d\sigma_q-1\right)\\
   &=I+II,
\end{align*}
where $p$ tends to $p_0$ along the straight line which is vertical to $\partial\mathbb{D}^{n+1}_+$.
We firstly consider the first term. For any $q\in\partial\mathbb{D}^{n+1}_+$, there exists an element $\tau\in\textnormal{Isom}(\mathbb{D}^{n+1})$ such that
 (see Figure \ref{fig:05})
\begin{enumerate}
    \item $\tau(q)$=0;
    \item $\tau(\mathbb{D}^{n+1}_+)=\mathbb{D}^{n+1}_+$ and $\tau(\partial\mathbb{D}^{n+1}_+)=\partial\mathbb{D}^{n+1}_+$.
\end{enumerate}
\begin{figure}[ht]
\centering
\includegraphics[width=0.8\textwidth]{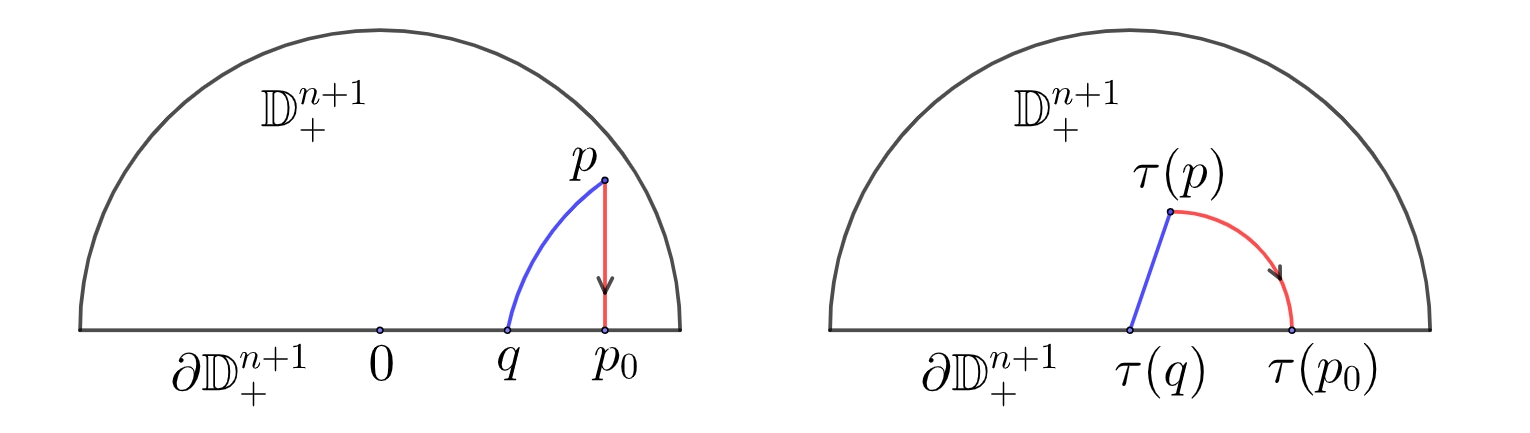}
\caption{}
\label{fig:05}
\end{figure}
Then $\tau(p)$ tends to $\tau(p_0)$ along a curve which is vertical to $\partial\mathbb{D}^{n+1}_+$.  It follows that
\begin{align}\label{e-I-1}
    \textnormal{dist}(\tau(p),\tau(p_0))\sim\frac{2y(\tau(p))}{1-|\tau(p)|^2}\text{ as }p\to p_0.
\end{align}
On the other hand, we have
\begin{align}\label{e_II-1}
   P(p_0,q)&=P(\tau(p_0),\tau(q)) \\
            &=-\frac{2}{\textnormal{Vol}(S^n)}\cdot\frac{1}{\sinh^{n}\textnormal{dist}(p_0,q)}\cdot\frac{\partial\ \textnormal{dist}(\tau(p_0),\tau(r))}{\partial\nu}|_{r=q}\nonumber\\
            &=-\frac{2}{\textnormal{Vol}(S^n)}\cdot\frac{1}{\sinh^{n}\textnormal{dist}(p_0,q)}\cdot\langle\nabla \textnormal{dist}(\tau(p_0),\tau(q)),\nu\rangle_{ds^2}\nonumber\\
            &=\frac{2}{\textnormal{Vol}(S^n)}\cdot\frac{1}{\sinh^{n}\textnormal{dist}(p_0,q)}\cdot\frac{y(\tau(p))}{|\tau(p)|}.\nonumber
\end{align}
From \eqref{e-I-1} and \eqref{e_II-1}, we have for any $q\in\partial\mathbb{D}^{n+1}_+$, as $p\to p_0$,
\begin{align*}
    \frac{P(p,q)}{\textnormal{dist}(p,q_0)}&=\frac{P(\tau(p),\tau(q))}{\textnormal{dist}(\tau(p),\tau(q_0))}\\
    &\sim \frac{1-|\tau(p)|^2}{2y(\tau(p))}\times \frac{2}{\textnormal{Vol}(S^n)}\cdot\frac{1}{\sinh^{n}\textnormal{dist}(p_0,q)}\cdot\frac{y(\tau(p))}{|\tau(p)|}\\
    &\sim\frac{1-|\tau(p_0)|^2}{|\tau(p_0)|}\cdot\frac{1}{\textnormal{Vol}(S^n)}\cdot\frac{1}{\sinh^{n}\textnormal{dist}(p_0,q)}\\
    &= \frac{1}{\textnormal{Vol}(S^n)}\cdot\frac{1}{\sinh^{n+1}\textnormal{dist}(p_0,q)},
\end{align*}
the last equality holds since $|\tau(p)|=\tanh\frac{\textnormal{dist}(\tau(p),0)}{2}$. Hence
\begin{align}\label{e-I}
    I=\frac{2}{\textnormal{Vol}(S^{n})}P.V.\int_{\partial\mathbb{D}^n_+}\frac{f(q)-f(p_0)}{\sinh^{n+1}\textnormal{dist}(p_0,q)}d\sigma_{q}.
\end{align}
For the second term, according to the conditions $(a),\ (b)$ in Page 21 and \eqref{har-ext}, we have
\begin{align*}
    \int_{\partial\mathbb{D}^{n+1}_+}P(p,q)d\sigma_q&=\varphi_n(p)\\
    &=1-c_n\int_0^r\frac{1}{\cosh^n s}ds
\end{align*}
where $r=\textnormal{dist}(p,\partial\mathbb{D}^{n+1}_+)$. Hence
\begin{align}\label{e-II}
II&=\lim\limits_{p\to p_0}f(p_0)\cdot\frac{\varphi_n(p)-1}{\textnormal{dist}(p,p_0)}\\
&=\lim\limits_{p\to p_0}-c_nf(p_0)\frac{\textnormal{dist}(p,\partial\mathbb{D}^{n+1}_+)}{\textnormal{dist}(p,p_0)}=-c_nf(p_0).\nonumber
\end{align}
Combine with \eqref{e-I} and \eqref{e-II}, we have the following theorem.
\begin{theorem}
    The DtN operator $\mathcal{D}_{\mathbb{D}_+^{n+1}}$ on $\mathbb{D}^{n+1}_+\ (n\geq 1)$ could be represented as follows: for any
    $f\in C_c^\infty(\partial\mathbb{D}^{n+1}_+)$ and $p_0\in\partial\mathbb{D}^{n+1}_+$,
\begin{align*}
    \mathcal{D}_{\mathbb{D}_+^{n+1}}(f)(p_0)=c_nf(p_0)-\frac{2}{\textnormal{Vol}(S^{n})}P.V.\int_{\partial\mathbb{D}^{n+1}_+}\frac{f(q)-f(p_0)}{\sinh^{n+1}\textnormal{dist}(p_0,q)}d\sigma_{q},
\end{align*}
where $\textnormal{Vol}(S^n)$ is the Euclidean volume of $n-$dimension sphere and 
$$c_n=\left(\int_0^\infty\frac{1}{\cosh^n s}ds\right)^{-1}.$$
\end{theorem}

\noindent In particular, for the case of $n+1=2$, for any $f\in C^\infty_c(\partial\mathbb{D}^2_+)$,
$$\mathcal{D}_{\mathbb{D}_+^{2}}f(t)=c_1f(t)-\frac{1}{\pi}P.V.\int_{-1}^1\frac{f(t)-f(s)}{\sinh^{2}\textnormal{dist}(t,s)}\cdot\frac{2}{1-s^2}ds.$$
Set 
$$s=\frac{e^x-1}{e^x+1},\ t=\frac{e^y-1}{e^y+1}\text{ and }g(u)=f\left(\frac{e^u-1}{e^u+1}\right),$$
then we have $g\in C_c^\infty((-\infty,\infty))$ and 
\begin{align*}
\left(\mathcal{D}_{\mathbb{D}^2_+}f,f\right)_{L^2(\partial\mathbb{D}^2_+)}&=c_1||f||^2_{L^2(\partial\mathbb{D}^2_+)}\\
&-\frac{1}{\pi}P.V.\int_{-1}^{1}\int_{-1}^1 \frac{(f(t)-f(s))f(s)}{\sinh^{2}\textnormal{dist}(t,s)}\cdot\frac{2}{1-s^2}\cdot\frac{2}{1-t^2}dsdt\\
&=c_1||f||^2_{L^2(\partial\mathbb{D}^2_+)}-\frac{1}{\pi}P.V.\int_{-\infty}^{\infty}\int_{-\infty}^\infty \frac{g(x)-g(y)}{\sinh^2(x-y)}g(y)dxdy\\ 
    &=c_1||f||^2_{L^2(\partial\mathbb{D}^2_+)}+\frac{1}{2\pi}\int_{-\infty}^{\infty}\int_{-\infty}^\infty \frac{(g(x)-g(y))^2}{\sinh^2(x-y)}dxdy\\
    &\leq c_1||f||^2_{L^2(\partial\mathbb{D}^2_+)}+\frac{1}{2\pi}\int_{-\infty}^{\infty}\int_{-\infty}^\infty \frac{(g(x)-g(y))^2}{|x-y|^2}dxdy\\
    &=c_1||f||^2_{L^2(\partial\mathbb{D}^2_+)}+C\cdot(\sqrt{-\Delta_{\mathbb{R}}}g,g)_{L^2(\mathbb{R})}
\end{align*}
   for some universal constant $C>0$. Since $$\inf \textnormal{Spec}\sqrt{-\Delta_{\mathbb{R}}}=0,$$ it follows that for any $\epsilon>0$, there exists $f\in C^\infty_c((-1,1))$ such that 
\begin{align*}
   (\sqrt{-\Delta_{\mathbb{R}}}g,g)_{L^2(\mathbb{R})}&\leq\epsilon\cdot \int_{-\infty}^{\infty}g(x)^2dx\\
    &=\epsilon\cdot \int_{-1}^1 \frac{2}{1-x^2}f(x)^2dx=\epsilon\cdot||f||^2_{L^2(\partial\mathbb{D}^2_+)}.
\end{align*}
Hence
$$\left(\mathcal{D}_{\mathbb{D}^2_+}f,f\right)_{L^2(\partial\mathbb{D}^2_+)}\leq\left(c_1+\epsilon\right)||f||^2_{L^2(\partial\mathbb{D}^2_+)}$$
for any $\epsilon>0$, which implies that
$$\inf\textnormal{Spec}\left(\mathcal{D}_{\mathbb{D}^2_+}\right)\leq c_1.$$
It follows that
\begin{align*}
\inf\textnormal{Spec}\left(\mathcal{D}_{\mathbb{D}^2_+}\right)&=c_1\\
&=\left(\int_0^\infty \frac{1}{\cosh t}dt\right)^{-1}\\
&=\left(\arctan \sinh t|_0^\infty\right)^{-1}=\frac{2}{\pi}.
\end{align*}
In summary, we have the following theorem.
\begin{theorem}
    For $n\geq 3$,
    $$\inf\textnormal{Spec}\left(\mathcal{D}_{\mathbb{D}^n_+}\right)\geq \left(\int_0^\infty\frac{1}{\cosh^{n-1} s}ds\right)^{-1},$$
    and for the case that $n=2$,
    $$\inf\textnormal{Spec}\left(\mathcal{D}_{\mathbb{D}^2_+}\right)=\frac{2}{\pi}.$$
\end{theorem}
\noindent It is natural to ask the following question
\begin{que*}
    What is the exact value of $\inf\textnormal{Spec}\left(\mathcal{D}_{\mathbb{D}^n_+}\right)$ for $n\geq 3$?
\end{que*}

\noindent{\bf Acknowledgments.}
 The first author is supported by NSFC, No. 12371056 and Shanghai Science and Technology Program. No. 22JC1400100.

\bibliographystyle{plain}
\bibliography{ref}
 
\noindent {Bobo Hua\\
School of Mathematical Sciences, LMNS,\\
Fudan University, Shanghai 200433, P.R. China\\
Shanghai Center for Mathematical Sciences, Jiangwan Campus, \\
Fudan University, No. 2005 Songhu Road, Shanghai 200438, P.R. China
\\
e-mail: bobohua@fudan.edu.cn }
\medskip
\\
\noindent {Yang Shen\\
School of Mathematical Sciences, \\
Fudan University, Shanghai, 200433, P.R. China\\
e-mail: shenwang@fudan.edu.cn }
\end{document}